\documentclass[12pt]{article}

\usepackage[utf8]{inputenc}

\usepackage{amssymb}
\renewcommand{\geq}{\geqslant}
\renewcommand{\leq}{\leqslant}

{\clearpage\newpage\section*{Corrigendum -- submitted #1}}{}

\ifx\volno\undefined\def\volno{}\fi
\ifx\volyear\undefined\def\volyear{}\fi
\ifx\papno\undefined\def\papno{}\fi

\newfont{\footsc}{cmcsc10 at 8truept}
\newfont{\footbf}{cmbx10 at 8truept}
\newfont{\footrm}{cmr10 at 10truept}
  \newlength{\BiblioSpacing}
  \renewenvironment{thebibliography}[1]{%
    \begin{oldthebibliography}{#1}%
      \setlength{\parskip}{\BiblioSpacing}
      \setlength{\itemsep}{\BiblioSpacing}
  }%
  {%
    \end{oldthebibliography}%
  }

\usepackage{amsthm,amsmath,amssymb}
\usepackage{graphicx}
\usepackage{cite}
\usepackage[colorlinks=true,citecolor=black,linkcolor=black,urlcolor=blue]{hyperref}

\theoremstyle{plain}
\newtheorem{theorem}{Theorem}[section]
\newtheorem{lemma}[theorem]{Lemma}
\newtheorem{corollary}[theorem]{Corollary}
\newtheorem{proposition}[theorem]{Proposition}

\theoremstyle{definition}
\newtheorem*{definition}{Definition}
\newtheorem{fact}{Fact}

\theoremstyle{remark}

\def\R{\mathbb{R}}

\def\N{\mathbb{N}}

\def\parts{\operatorname{par}}

\def\len{\operatorname{len}}
\def\fst{\operatorname{fst}}
\def\d{\mbox{-}}

\title{\bf Counting words with Laguerre series}

\author{Jair Taylor\\
\small Department of Mathematics\\[-0.8ex]
\small University of Washington\\[-0.8ex] 
\small Seattle, WA, U.S.A.\\
\small\tt jptaylor@uw.edu\\
}

\begin{document}

\maketitle
\begin{abstract}
We develop a method for counting words subject to various restrictions by finding a combinatorial interpretation for a product of weighted sums of Laguerre polynomials with parameter $\alpha = -1$.  We describe how such a series can be computed by finding an appropriate ordinary generating function and applying a certain transformation. We use this technique to find the generating function for the number of $k$-ary words avoiding any vincular pattern that has only ones, as well as words cyclically avoiding vincular patterns with only ones whose runs of ones between dashes are all of equal length.
\end{abstract}
\maketitle\section{Introduction}

Consider the following simple problem.  How many arrangements of the word ``WALLAWALLA'' are there with no LLL, AAA or WW as consecutive subwords?  Perhaps surprisingly, the answer can be calculated by performing a certain integral.  In fact, there are \[\int_0^\infty e^{-t}\left(\frac{1}{24}t^4 - t^2 + t\right)^2  \left(\frac{1}{2}t^2 - t\right) \, dt = 1584\] such words, as we will see.  More generally, we will develop a technique by which the generating function for the number of words subject to various conditions can be found by integrating a product against $e^{-t}$, $t\geq 0$.  Since $\int_0^\infty e^{-t}t^n \, dt = n!$, this amounts to applying the linear functional $\Phi: \R[t] \mapsto \R$ given by $t^n \mapsto n!$.

Define a {\it word} $W$ on an alphabet $S$ to be an ordered list $s_1 \cdots s_n$ of letters $s_i \in S$. A subword of $W$ is a word $s_l s_{l+1} \cdots s_{l+m}$. Note that we require the indices in a subword to be consecutive, while some authors do not.  Define a {\itshape factorization} $\phi$ on $S$ to be an ordered list $(\phi_1)(\phi_2) \ldots (\phi_k)$ of non-empty words $\phi_1, \ldots, \phi_k$ using letters from $S$.  We call $\phi_1, \ldots, \phi_k$ the {\it factors} or {\it parts} of $\phi$.  We abuse notation slightly by identifying a word $W$ with the factorization $(W)$ in one part, and the empty word with the factorization having no parts, writing $\emptyset$ for both.  

Given a factorization $\phi$ on $S$ and $T \subseteq S$, let $\phi|_T$ be the factorization whose parts are the maximal subwords of the parts of $\phi$ that use only letters from $T$, the parts of $\phi|_T$ being ordered by their occurrence in $\phi$.  We call $\phi|_T$ the restriction of $\phi$ to $T$.  For example, if $S = \{a,b\}$ and $T = \{a\}$, then the restriction of the factorization $(aabba)(aab)(b)(aaa)$ to $T$ is $\phi|_T = (aa)(a)(aa)(aaa)$, while $\phi|_{\{b\}} = (bb)(b)(b)$.  If $\phi$ contains no letters from $T$, we define $\phi|_T$ to be the empty factorization.

The methods given here will allow us to count words where only certain restrictions can occur.  For example, consider the words $W$ so that for each $s \in S$, the factors of $W|_{\{s\}}$ are all length one. This is equivalent to requiring that $W$ have no adjacent, identical letters.  Such words are sometimes called {\it Carlitz}, after Leonard Carlitz.  Thus ``PEAR'' is Carlitz, but ``APPLE'' is not.

More generally, let $A_1$ and $A_2$ be two sets of factorization on disjoint alphabets $S_1$, $S_2$, respectively.  Let $S = S_1 \cup S_2$, and denote by $A_1* A_2$ the set of factorizations $\phi$ of words on $S$ so that $\phi|_{S_1} \in A_1$ and $\phi|_{S_2} \in A_2$.  For example, if $A_1 = \{\emptyset, (a),(a)(a),(a)(a)(a), \ldots\}$, $B_2 = \{\emptyset, (b),(b)(b),(b)(b)(b), \ldots\}$, then $A_1 * A_2$ is the set of factorizations on $\{a,b\}$ so that each factor is a Carlitz word.  Thus $(ab)(a)(a)(baba) \in A_1 * A_2$, but $(b)(aab)(ba) \notin A_1 *A_2$.

Perhaps surprisingly, the problem of counting words in $A_1*A_2$ can be reduced to multiplication of certain power series.  If $A$ is a set of factorizations, we define a weight $w$ on $A$ to be a polynomial-valued function on $A$ that obeys a certain combinatorial restriction.  Then we will define a power series $f_{A,w}(t)$, the associated {\it Laguerre series} for $A$, in terms of the generalized Laguerre polynomials with parameter $\alpha = -1$.  The key fact we will use is the rule (Theorem~\ref{key}) \[f_{A_1*A_2,w}(t) = f_{A_1,w}(t) \cdot f_{A_2,w}(t).\]  Furthermore, we will show in Proposition~\ref{phifact} that $\Phi(f_{A,w}(t))$ gives the weight of all factorizations in $A$ with one or no parts, which we identify with words. By applying $\Phi$ to a product of Laguerre series we may count a variety of sets of restricted words, especially when the restrictions are on the length of runs of particular letters.

In Section~3, we describe the transformation $T$ that turns certain ordinary generating functions into Laguerre series.  The transformation can be described in terms of the Laplace transform, and so can be easily implemented in mathematics software packages.  We use $T$ to determine the Laguerre series for different sets of factorizations $A$, and use them to derive formulas and generating functions which count words that obey various restrictions.

In particular, we can use this technique to analyze certain pattern avoidance problems. A vincular, or generalized, pattern is a pattern with dashes such as $13 \d 2$.  This is a generalization of classical permutation patterns where the dashes are used to indicate that the numbers on either side are not required to be adjacent, but all others are.  These patterns were first studied by Babson and Steingr\'{i}msson \cite{Steingrimsson}, who showed that many statistics of interest can be classified in terms of vincular patterns.  The term {\itshape vincular} itself was coined by Claesson in \cite{Claesson}, from the Latin {\it vinculare}, to bind.  Words avoiding vincular patterns are studied in \cite{Bernini,Burstein, Mansour2,Mansour1,Mansour5, Mansour6}.  In this paper we will study vincular patterns with all ones, such as $111\d 11$.  Although such patterns are useless in the context of permutations, where only the pattern $1$ can be contained, they are meaningful in the context of general words on the alphabet $\N$ where letters may be repeated.  Since all of our patterns will have only ones, we give the definition in this context.  A word $W = s_1 \cdots s_l$, with each $s_i$ in some alphabet $S$, contains a vincular pattern $\tau = 1^{m_1}\d\cdots\d 1^{m_n}$ if there is a subsequence of $W$ consisting of $m = m_1 + \ldots + m_n$ identical letters of which the first $m_1$ are consecutive, the next $m_2$ are consecutive, and so on.  For example, a word contains $111\d 11$ if it has five appearances of the same letter in the word so that the first three and the last two are consecutive.  Formally, we say that $W$ contains $\tau$ if there are indices $1 \leq i_1 < i_2 < \ldots < i_m \leq l$ with $s_{i_1} = \ldots = s_{i_m}$ and $i_{j+1} - i_j = 1$ for $j \neq m_1, m_1 + m_2, \ldots, m_1 + \ldots + m_{n-1}$.  Otherwise, we say that $W$ avoids $\tau$.    

In Section~$4$, we give a formula to calculate the generating function for the number of words avoiding any such vincular pattern with only ones.  This formula involves the use of the maps $T$ and $\Phi$, but these can be easily calculated.  For example, we can use Sage to compute the the generating function $\sum_W x^{\len(W)}$ where the sum is taken over all ternary words $W$ avoiding the pattern $11  \d 11$, where $\len(W)$ is the ${\it length}$ of $W$, the number of letters counting multiplicity:

\[\frac{6 \, x^{7} - 6 \, x^{6} + 6 \, x^{5} - 2 \, x^{4} - 5 \, x^{3} + 9 \, x^{2} - 5 \, x + 1}{16 \, x^{4} - 32 \, x^{3} + 24 \, x^{2} - 8 \, x + 1} = 1 + 3 x + 9 x^{2} + 27 x^{3} + 78 x^{4} + 222 x^{5}+ \ldots.\]

We  give a cyclic version of this result in Section~$5$ for the case of patterns $1^m \d 1^m \d \cdots \d 1^m$, where all runs of ones are the same length.  This gives the generating functions for words so that any cyclic permutation of their letters avoids such a pattern.  This generalizes a result of Burstein and Wilf \cite{Wilf} which gives the generating function for the number of words cyclically avoiding $1^m$.

\section*{Acknowledgments}
The author would like to thank Sara Billey, Ira Gessel, Silvia Heubach, Yannick van Huele, Toufik Mansour, Brendan Pawloski, Austin Roberts, Byron Schmuland and William Stein for their extraordinary support.  Without them this paper would still be ``floating in platonic heaven''.

\section{Laguerre polynomials and Laguerre series} 
Define the polynomials $l_k(t)$ by their generating function 
\begin{align}\label{laguerregf}
\sum_{k=0}^\infty l_k(t) x^k = e^{\frac{tx}{1+x}}.
\end{align} The first few such polynomials are
\begin{align*}
l_0(t)  &= 1\\
l_1(t) &= t\\
l_2(t) &= \frac{1}{2}t^2 - t\\
l_3(t) &= \frac{1}{6}t^3 - t^2 + t.
\end{align*}
These polynomials are a form of Laguerre polynomial.  Specifically, $l_k(t) = (-1)^k L_k^{(-1)}(t)$ where $$L_k^{(\alpha)}(t) = \sum_{i=0}^k (-1)^i \binom{k + \alpha }{k - i} \frac{t^i}{i!}$$ defines the generalized Laguerre polynomials. They have been found to have a number of interesting combinatorial properties, beginning with their use by Even and Gillis to count generalized derangements when $\alpha$ is set to $0$ in~\cite{Even}.  This was later extended by Foata and Zeilberger who use $\alpha$ to keep track of the number of cycles \cite{Doron3}.  For our purpose, we will take $\alpha = -1$.

Recall that a word on $[k] = \{1,2, \ldots, k\}$ is called $k$-ary, and a word in which no two adjacent letters are the same is called a Carlitz word.  In 1988, Ira M. Gessel \cite[Section 6]{Gessel} used his theory of generalized rook polynomials to find an explicit formula for the number of Carlitz arrangements of a given multiset of letters.  We present here an unlabeled version.

\begin{theorem}\label{fundamental}
Let $\Phi$ be the linear functional on polynomials in $t$ given by $\Phi(t^n) = n!$.  Given nonnegative integers $n_1, \ldots, n_k$, the number of $k$-ary Carlitz words with the letter $i$ used exactly $n_i$ times is \[\Phi\biggl(\prod_{i=1}^k l_{n_i}(t)\biggr).\]
\end{theorem}

For example, in \cite{Mississippi} the authors consider the ``Mississippi Problem''.  How many arrangements of the letters in the word ``MISSISSIPPI'' have no adjacent letters the same?  We can use the preceding theorem to calculate this directly.  There is one $M$, four $I$'s, four $S$'s, and two $P$'s.  So the solution is \[\Phi(l_1(t)l_4(t)l_4(t)l_2(t)) = \int_0^\infty e^{-t} \, \Big( t \Big) \, \left( \frac{1}{24}t^4 - \frac{1}{2}t^3 + \frac{3}{2} t^2 - t\right)^2\left(\frac{1}{2}t^2 - t\right)\, dt = 2016.\]

Our proof of Theorem~\ref{fundamental} does not involve rook polynomials, but is based on the following surprisingly simple expression for the generating function for the number of Carlitz words with a given set of letters.

\begin{lemma}\label{carlitzgf}
Let $\mathcal{C}$ be the set of Carlitz words over the alphabet $\N$.  For a word $W = s_1\cdots s_n$ define its weight to be $w(W) = x_{s_1}\cdots x_{s_n}$, so that the exponent of $x_i$ records the number of occurrences of the letter $i$.  Then 
\begin{align*}\label{smirnov}
\sum_{W \in \mathcal{C}} w(W) = \frac{1}{1-\sum_i \frac{x_i}{1+x_i}}.
\end{align*}
\end{lemma}
Lemma ~\ref{carlitzgf} can be found in the book by Goulden and Jackson \cite[2.4.16]{Goulden-Jackson2}. They attribute it MacMahon's book from 1915 \cite{MacMahon} and Smirnov, Sarmanov, and Zaharov in (1966) \cite{Smirnov}, and so counting Carlitz words is sometimes known as the Smirnov problem.  The case $x_i = x^i$ was found, apparently independently, by Carlitz in 1977 \cite{Carlitz}.  Asymptotics of Carlitz compositions are investigated using this formula by Knopfmacher and Prodinger in \cite{Knopfmacher}. This result can be generalized to the Carlitz-Scoville-Vaughan Theorem \cite{CSV}, a reciprocity result relating the words for which only a given set of pairs of letters are adjacent to the words where none of these pairs are adjacent. 

\begin{proof}[Proof of ~\ref{carlitzgf}]
We have
\begin{align*}
\frac{1}{1-\sum_i \frac{x_i}{1+x_i}} &= \sum_{n=0}^\infty \left(\sum_i \frac{x_i}{1+x_i}\right)^n\\
&= \sum_{n=0}^\infty \left(\sum_i \left( x_i -x_i^2 + x_i^3 - x_i^4 + \ldots\right)\right)^n .\\
\end{align*}
We see that this sum represents a method of choosing a word $W$. We choose a number of steps, $n$, and at each step we choose a letter $i$ and append any number of copies of it.  Suppose $W$ is not Carlitz, and write $W = W_1iiW_2$ where $W_1$ is Carlitz so that the subword $ii$ is the first adjacent repetition. If the $i$'s were chosen in separate steps we may merge the steps together, and if the $i$'s were chosen in the same step we may split the steps apart.  This gives a bijection between the ways $W$ may be formed in these two cases, and the monomials representing them in the power series above will appear with opposite sign. Thus any such $W$ will be canceled in the sum.  The remaining Carlitz words can only be represented in one way, with positive sign.
\end{proof}

Now we can give a simple proof of Gessel's formula.

\begin{proof}[Proof of Theorem~\ref{fundamental}]
Note that \[\Phi(e^{tf}) = \frac{1}{1-f}\] for any $f$ that is constant with respect to $t$ when both sides are defined, where $\Phi$ is extended in the natural way to some power series.  Then we see
\begin{align*}
\Phi\left(\prod_{i=1}^\infty \exp\left(\frac{tx_i}{1+x_i}\right)\right) &= \Phi\left( \exp\left(t\sum_{i=1}^\infty\frac{x_i}{1+x_i}\right)\right)\\
&= \frac{1}{1 - \sum_{i=1}^\infty \frac{x_i}{1+x_i}}\\
\end{align*}
and if we take the coefficient of $x_1^{k_1} x_2^{k_2} \cdots x_m^{k_m}$ on both sides we get the desired formula by Lemma ~\ref{carlitzgf} and the generating function \eqref{laguerregf}.
\end{proof}

Using Theorem~\ref{fundamental}, it is easy to see combinatorially that 

\begin{align}\label{orthogonality}
  \Phi(l_i(t) l_j(t))  = \left\{
     \begin{array}{ll}
       2 & \text{if $i = j$} \\
       1 & \text{if $|i-j| = 1$}\\
		 0 & \text{if $|i-j| > 1$}\\
     \end{array}
   \right.
\end{align} 

\noindent and so the polynomials $l_k(t)$ are ``almost'' orthogonal with respect to $\Phi$.  

Note that $l_k(t)$ is a polynomial of degree $k$; so the matrix of $l_k$'s expanded into powers of $t$ is triangular with no zeroes on the diagonal, and so $\{l_k\}_k$ forms a basis of $\R[t]$.  It is natural to ask, then, what is the expansion of $l_i(t) l_j(t)$ in this basis?  These are known as {\it linearization coefficients}.  The linearization coefficients of general Laguerre polynomials, with $\alpha$ indeterminate, is known \cite{ Doron3,zeng}, but we will need a combinatorial interpretation of the case $\alpha = -1$.

Denote by $n_{i,j,k}$ the number of factorizations over the alphabet $\{a,b\}$ with $k$ parts and exactly $i$ $a$'s and $j$ $b$'s so that each part is Carlitz.  For example, $n_{2,5,3} = 6$: the possibilities are $(bab)(bab)(b)$, $(babab)(b)(b)$ and the different permutations of these sets of factors.  

\begin{lemma}\label{linearization}
We have, for all $i,j \in \N$, \[l_i(t)l_j(t) = \sum_k n_{i,j,k}l_k(t).\]
\end{lemma}
\begin{proof}
Note that if $p(t) = a_0 + a_1t + \ldots +a_nt^n$ is a polynomial and $\Phi(t^m p(t)) = 0$ for all $m$, then \[a_0m! +a_1(m+1)! + \ldots + a_n(n+m)! = 0.\] This is a homogenous linear recurrence relation with constant coefficients for the factorial sequence, which is impossible unless $a_0 = a_1 = \ldots = a_n = 0$ since it the factorial is superexponential.  Since $\{l_k(t)\}_k$ forms a basis for $\R[t]$, if $\Phi(p(t)l_k(t)) = 0$ for all $k$ then we can conclude $p(t) = 0$.  So it is enough to show that for any nonnegative integers $i,j,m$,

\[\Phi\left( l_i(t)l_j(t) l_m(t) \right) = \phi\left(\sum_k n_{i,j,k} l_k(t) l_m(t) \right).\]

We know that the left hand side counts the number of Carlitz arrangements of $i$ $a$'s, $j$ $b$'s, and $m$ $c$'s, while the right hand side gives the total number of pairs $(\phi, W)$ where $\phi$ is a factorization in $k$ parts with $i$ $a$'s and $j$ $b$'s with each part Carlitz, and $W$ is a Carlitz word with $k$ $x$'s and $m$ $c$'s. There is a simple bijection between these sets.  Given such a pair $(\phi, W)$, we can get a Carlitz arrangement of $i$ $a$'s, $j$ $b$'s and $m$ $c$'s by replacing the $i$th $x$ of $W$ with the $i$th part of $\phi$.  For example, if $\phi = (ab)(bab)$ and $W = cxcx$, we get the Carlitz word $cabcbab$.  This process is reversible:  given a Carlitz word on $a,b,c$ we replace the $c$'s by parentheses to make a factorization $\phi$ with only the letters $a$ and $b$, and to get $W$ we replace each maximal subword that does not contain $c$ by a single $x$, getting a word with only $c$'s and $x$'s.  For example, given the word $abcbcab$, we get the pair $W = xcxcx$, and $\phi = (ab)(b)(ab)$.  The maximality condition guarantees that $W$ will be Carlitz.\end{proof}

\begin{definition} Given a set of factorizations $A$ on an alphabet $S$, a {\it weight} is a function $w$ from $A$ and all of the restrictions of factorizations in $A$ into a polynomial ring $\R[x_1, x_2, \ldots]$ that commutes with restriction in the sense that if $\phi \in A$ and $T \subseteq S$, then $w(\phi) = w(\phi|_T) \, w(\phi|_{S\backslash T})$.
\end{definition}

Note that in particular, if $A = A_1 *A_2$ for some sets of factorizations $A_1,A_2$ then $w$ is also a weight on $A_1$ and $A_2$.  Also note that taking $T$ to be empty forces $w(\emptyset) = 1$.  Typically we will take the weight $w(\phi)$ to be a monomial $x_1^{n_1(\phi)} x_2^{n_2(\phi)} \cdots x_m^{n_m(\phi)}$ where each $n_i(\phi)$ is a statistic so that $x_i^{n_i(\phi)}$ is multiplicative in the above sense.  In all our examples it will be obvious that $w$ is a weight.  Examples include 

\begin{itemize}
\item $\len(\phi)$, the length of $\phi$
\item $\operatorname{sum}(\phi)$, the sum of the letters in $\phi$ if the symbols in $\phi$ are nonnegative integers
\item the number of distinct symbols in $\phi$
\item the number of appearances of a particular symbol
\end{itemize}

\noindent or simply $w=1$ if we wish to enumerate a finite set.  We write $\parts(\phi)$ for the number of parts of $\phi$; but $x^{\parts(\phi)}$ is in general {\it not} a weight.

\begin{definition} Let $A$ be a set of factorizations on an alphabet $S$ and $w$ be a weight on $A$.  Define the {\it Laguerre series} of $A$ with respect to $w$ to be the formal power series \[f_{A,w}(t) = \sum_{\phi \in A} w(\phi) l_{\parts(\phi)}(t)\] when this sum is well-defined as a formal power series.  For convenience we will omit the $w$ in the subscript when $w=1$, writing $f_{A,1}(t)$ as $f_A(t)$.
\end{definition}
Convergence of Laguerre series is studied in, e.g., \cite{Pollard, szasz, weniger}, where they are defined as series $$\sum_n \lambda_n^{(\alpha)}L_n^{(\alpha)}(t).$$  Our definition is slightly different, as we always take $\alpha = -1$ and our coefficients may be weighted.


\begin{proposition}\label{phifact}
Assume $A$ is a set of factorizations and $w$ is a weight on $A$.  Let $\Phi$ be the linear operator so that $\Phi(t^n) = n!$ and $\Phi$ fixes any other variables.  Then \[\Phi( f_{A,w}(t)) = \sum_W w(W)\] when both sides are defined, where the sum is over allowed words $W \in A$ which are factorizations with one or no parts.
\end{proposition}

\begin{proof}
We have
\begin{align*}
\Phi(f_{A,w}(t)) &= \sum_{\phi \in A} w(\phi) \Phi( l_{\parts(\phi)}(t))\\
\end{align*}
and $\Phi( l_{\parts(\phi)}(t))$ is $1$ when $\phi$ has one or no parts and is $0$ otherwise by~\eqref{orthogonality} since $l_0(t) = 1$ and $l_1(t)= t$.
\noindent 
\end{proof}

It is easy to see that if $A_1$ and $A_2$ are disjoint sets of allowed factorizations on a common alphabet $S$, $w$ is a weight on $A_1 \cup A_2$, then \[f_{A_1 \cup A_2,w}(t) = f_{A_1,w}(t) + f_{A_2,w}(t).\]  More interesting, perhaps, is the combinatorial interpretation of a product of Laguerre series.  This is our main theorem on the combinatorial properties of Laguerre series.   Recall that if $A_1,A_2$ are sets of factorizations on disjoint alphabets $S_1,S_2$, then $A_1* A_2$ was defined to be the set of factorizations $\phi$ of words on $S = S_1 \cup S_2$ so that $\phi|_{S_1} \in A_1$ and $\phi|_{S_2} \in A_2$. 

\begin{theorem} \label{key} 
Let $S_1$ and $S_2$ be disjoint alphabets with sets of allowed factorizations $A_1$,$A_2$ respectively, and let $w$ be a weight on $A_1 * A_2$ (and hence on $A_1$ and $A_2$.)  Then \[f_{A_1 * A_2,w}(t) = f_{A_1,w}(t)\cdot f_{A_2,w}(t).\]
\end{theorem}

\begin{proof}
By Lemma 2.3,
\begin{align*}
 f_{A_1,w}(t)\cdot f_{A_2,w}(t) &= \sum_{\phi_1 \in A_1} \sum_{\phi_2 \in A_2} w(\phi_1) w(\phi_2) l_{\parts(\phi_1)}(t) l_{\parts(\phi_1)}(t)\\
 &= \sum_{\phi_1 \in A_1, \phi_2 \in A_2, k\geq 0} n_{\parts(\phi_1), \parts(\phi_2),k} w(\phi_1) w(\phi_2) l_k(t).
\end{align*}

Fix $\phi_1 \in A_1, \phi_2 \in A_2$.  It is enough to show that $n_{\parts(\phi_1), \parts(\phi_2),k}$ is the number of factorizations with $k$ parts on $S_1 \cup S_2$ whose restrictions to $S_1$ and $S_2$ are $\phi_1$ and $\phi_2$, respectively.  Then each allowed word of $A_1 * A_2$ will then be represented exactly once in the series $f_{A_1,w}(t)\cdot f_{A_2,w}(t)$, giving \[f_{A_1 * A_2,w}(t) = \sum_{\phi \in A_1 * A_2} w(\phi) l_{\parts(\phi)} =  f_{A_1,w}(t)\cdot f_{A_2,w}(t).\]
 
For fixed $k$, we will construct a simple bijection from the set of triples $(\phi,\phi_1, \phi_2)$ where $\phi_1$, $\phi_2$ are factorizations in $A_1,A_2$ respectively and $\phi$ is a factorization on the alphabet $\{a,b\}$ with $\parts(\phi_1)$ $a$'s and $\parts(\phi_1)$ $b$'s so that each part is Carlitz, and the set of factorizations $\phi_3$ of $A_1 * A_2$ with $k$ parts.  Let $\phi_3$ be the factorization created by replacing the $n$th $a$ in $\phi$ with the $n$th part of $\phi_1$, and the $n$th $b$ with the $n$th part of $\phi_2$.  Then by construction $\phi_3 \in A_1 * A_2$: its restrictions are $\phi_1$ and $\phi_2$.  Furthermore, given an allowed factorization $\phi_3 \in A_1 * A_2$ with $k$ parts so that $\phi_{S_1} = \phi_1$, $\phi_{S_2} = \phi_2$, we can reconstruct the factorization $\phi$ of a word on $\{a,b\}$ by replacing each subword of a factor of $\phi_3$ that uses only the letters of $S_1$, and is maximal with respect to this condition, by an $a$ and each maximal subword using only letters of $S_2$ by a $b$.  For example, if $S_1 = \{1,2\}$ and $S_2 = \{3,4\}$, with $\phi_3 = (123,2213,34413)$, we get the word $\phi = (ab,ab,bab)$.  No part of $\phi$ can have $aa$ or $bb$ by the maximality condition.  These two algorithms are inverse to each other, establishing the theorem.\end{proof}

Inductively, we see that if $A_1, \ldots, A_n$ are sets of factorizations on disjoint alphabets and $w$ is a weight on $A_1 * \cdots * A_n$, then $f_{A_1 *\cdots* A_n,w}(t) = f_{A_1,w}(t)\cdots f_{A_n,w}(t)$.  Thus we have the following extension of Lemma~\ref{linearization}:

\begin{corollary}
Let $S = \{a_1, \ldots, a_m\}$ be an alphabet of $m$ distinct letters, and $r_1, \ldots, r_m$ be fixed nonnegative integers.  Let $n_k$ denote the number of factorizations on $n$ letters in $k$ parts so that each part is Carlitz, and the $i$th letter $a_i$ of $S$ is used exactly $r_i$ times for each $i$.  Then \[\prod_{i=1}^m l_{r_i}(x) = \sum_{k\geq 0} n_k l_k(x).\]
\end{corollary}
\begin{proof}
Let $A_i$ consist only of the factorization $(a_i)(a_i) \cdots (a_i)$, with exactly $r_i$ factors.  Then $A_1 * \cdots * A_m$ consists of those factorizations for which each part is Carlitz and the letter $a_i$ is repeated $n_i$ times so we may apply Theorem~\ref{key}.  Alternately, we can adapt the proof of Lemma~\ref{linearization}.
\end{proof}

\section{Computing Laguerre series}

The Laguerre series for a set of factorizations would not be especially useful if it was difficult to compute.  Fortunately, there is an efficient method to calculate them in some situations.  It may be difficult to find a convenient formula for the coefficients of $l_k(t)$ in a given Laguerre series $f_{A,w}(t)$, but this is not needed to find an expression for $f_{A,w}(t)$.  It is enough to find the ordinary generating function.  Specifically, we define $$g_{A,w}(u) = \sum_{ \phi \in A} w(\phi) u^k.$$  If a nice form of $g_{A,w}(u)$ is known, we may obtain the Laguerre series $f_{A,w}(t)$ by applying the linear transformation $T$ that sends $u^k$ to $l_k(t)$.  As it happens, $T$ can be computed in many situations using the inverse Laplace transform.  We have \[ \mathcal{L}\{l_k(t)\} = \frac{1}{s(1-s)} \left( \frac{1-s}{s}\right)^k\] for $k\geq 1$, where $\mathcal{L}$ is the Laplace transform; this is easily proved from the formula for $l_k(t)$ in terms of the generalized Laguerre polynomials, the fact that $\mathcal{L}\{ t^i\} = \frac{i!}{s^{i+1}}$, and the binomial theorem.
 
Therefore, if $g_{A,w}(0) = 0$, we have by linearity \[ T\{g_{A,w}(u)\} = f_{A,w}(t) = \mathcal{L}^{-1}\left\{ \frac{g_{A,w}(\frac{1-s}{s})}{s(1-s)}\right\}\] when the right-hand side is well-defined.  If $g_{A,w}(0) \neq 0$, we can calculate  \[ f_{A,w}(t) = T\{g_{A,w}(u) - g_{A,w}(0)\} + g_{A,w}(0)  =  \mathcal{L}^{-1}\left\{ \frac{g_{A,w}(\frac{1-s}{s}) - g_{A,w}(0)}{s(1-s)}\right\} + g_{A,w}(0)\] since $l_0(t) = 1$.  The use of the inverse Laplace transform here is not central to the theory, but it is convenient since many software packages provide symbolic calculation of the inverse Laplace transform, making it easy to implement the transformation $T$.

In many cases, we can use standard generating function techniques to calculate $g_{A,w}(u)$.  For example, let $A$ be the set of all factorizations on a single-letter alphabet, and let $w$ be the weight $w(\phi)= x^{\len(\phi)}$.  Then we see that \[g_{A,w}(u) = \sum_{k=0}^\infty u^k(x+x^2 + \ldots)^k\] since we choose a factorization $\phi$ by choosing the number of parts and then choosing the size of each part.  We may simplify this expression to get $g_{A,w}(u) = \frac{1+x}{1 + x - ux},$ and calculate  $f_{A,w}(t) = T\{g_{A,w}(u)\} = e^{tx}$.  Taking the coefficient of $x^k$ shows that the Laguerre series for all factorizations on this single-letter alphabet with length $k$ is $\frac{t^k}{k!}$.  Thus Proposition~\ref{phifact} and Theorem~\ref{key} tell us that \[\Phi\left(\prod_{k=1}^n \sum_{i=0}^{r_k} \frac{t^i}{i!}\right)\] is the total number of words on an alphabet $S = \{c_1, \ldots, c_n\}$ so that $c_k$ is used at most $r_k$ times, where $r_1, \ldots, r_n$ is a set of nonnegative integers.  Of course, this formula can also be seen directly.

Laguerre series are sometimes useful for counting words with certain subwords prohibited.  Consider the following problem: given a word $W$ and a subword $W'$ of $W$, how many words are there avoiding $W'$ as a subword that can be made from $W$ using each symbol at most as many times as it appears in $W$?   To answer this, we form the set $A_1$ of factorizations that can be made from $W$ that use only letters from $W'$ and do not contain $W'$ as a subword of any factor.  Letting $A_2$ be the set of all factorizations of words that can be made from $W$ in the above sense that do not use any letters from $W'$, we see that $A_1 * A_2$ is the set of all factorizations of words that can be made from $W$ so that no part contains $W'$ as a subword.  Setting $w=1$, we may perform a brute force calculation to compute the Laguerre series $f_{A_1}(t)$.  Since there are no subword restrictions on the remaining letters, we have \[f_{A_2}(t) = \prod_c \sum_{i=1}^{n_c} \frac{t^i}{i!}\] where the product is over the letters $c$ used in $W$ that are not used in $W'$, and $n_c$ is the number of times $c$ is used in $W$.  So the solution is given by $\Phi(f_{A_1}(t)f_{A_2}(t))$.

This method is practical if there are not too many letters used in the $W'$, and the number of times these letters are repeated in $W$ is small.  For example, how many words are there that can be made from the letters in ``CONSTANTINOPLE'' that avoids the subword ``TNT''?  We focus on the letters $T$ and $N$, forming the set $A_1$ consisting of those factorization of words using only these letters, with three or fewer $N$'s and two or fewer $T$'s, so that no factor contains ``TNT''.  We then perform a brute force count using a programming language such as Sage to explicitly calculate the Laguerre series for $A_1$ with respect to $w=1$, getting \[f_{A_1}(t) = \frac{1}{12}t^5 + \frac{5}{12}t^4 + \frac{2}{3} t^3 + t^2 + t + 1.\]  Then forming $A_2$ as above, the set of factorizations of any words made from the remaining letters ``COSAIOPLE'', we calculate $f_{A_2}(t) = (t + 1)^7(\frac{1}{2}t^2 + t +1)$ since all of the letters appear only once, except O which appears twice.  Therefore, the answer is \[\int_0^\infty e^{-t} \left(\frac{1}{12}t^5 + \frac{5}{12}t^4 + \frac{2}{3} t^3 + t^2 + t + 1\right)\left(t +1\right)^7\left(\frac{1}{2}t^2 +t +1\right)\, dt = 9{,}854{,}474{,}467.\]

A more general method for the subword avoidance problem is the cluster method developed by Goulden and Jackson \cite{Goulden-Jackson1,Goulden-Jackson2}.  This a powerful tool using linear algebra and inclusion-exclusion to find rational generating functions for words avoiding any given set of subwords. Noonan and Zeilberger \cite{Doron1} give generalizations and a good introduction to the idea; Edlin and Zeilberger also give an extension to cyclic words \cite{Doron2}.

Now consider the problem of counting words that have no subword consisting of $m$ identical letters.  These are words that avoid the subword pattern $1^m$, and are sometimes called $m$-Carlitz words; when $m=2$ we have the ordinary Carlitz words.  To find the generating function, let $A$ be the set of factorizations on a one-letter alphabet with each part having length smaller than $m$, and again let $w(\phi)= x^{\len(\phi)}$.   We see that \[g_{A,w}(u) = \sum_{n = 0}^\infty u^n(x + \ldots + x^{m-1})^n =  \frac{1 - x}{1 -x - u(x - x^m)}\] and so we compute 
\begin{align}\label{lacefunforsmallruns}
f_{A,w}(t) = T\{g_{A,w}(u)\} = \exp\left(t \cdot\frac{x - x^m}{1 - x^m}\right).
\end{align} 

Taking the coefficient of $x^n$ in~\eqref{lacefunforsmallruns} gives the Laguerre series for the set of factorizations with length $n$ so that each part is smaller than $m$.  This gives a generalization of Theorem~\ref{fundamental}.  If $m_1, \ldots, m_k, n_1, \ldots, n_k$ are nonnegative integers, and $p_{m,n}(t)$ are polynomials defined by $\sum_{n=0}^\infty p_{m,n}(t) \,x^n = \exp{\left( \frac{t(x - x^m)}{1- x^m}\right)}$, we see that \[\Phi\left( \prod_{i =1}^k p_{m_i,n_i}(t)\right)\] is the total number of $k$-ary words that use the letter $i$ exactly $n_i$ times and do not contain the subwords $i^{m_i}$. Thus the number of arrangements of the word ``WALLAWALLA'' with no LLL, AAA or WW as consecutive subwords is \begin{align*}
\int_0^\infty e^{-t} p_{3,4}(t)\cdot p_{3,4}(t) \cdot p_{2,2}(t)\, dt &= \int_0^\infty e^{-t}\left(\frac{1}{24}t^4 - t^2 + t\right)^2 \cdot \left(\frac{1}{2}t^2 - t\right) \, dt\\
&= 1584.
\end{align*}

Recalling again the formula \[\Phi(e^{tf}) = \frac{1}{1-f},\]\label{phirule}we see that the generating function for the number of $k$-ary $m$-Carlitz words of length $n$ is given by \[\Phi\left(\exp\left( kt \cdot\frac{ x - x^m}{1-x^m} \right)\right) = \frac{1 - x^m}{1 - kx - (k-1)x^m}.\]  Another derivation of this formula is given by Burstein and Mansour \cite[Example~2.2]{Mansour2}.

We may set indeterminates to count the number of occurrences of each symbol. Let $w(\phi) = x_1^{i_1} x_2^{i_2} \cdots x_k^{i_k}$, where $i_j$ is the number of appearances of $j$ in a factorization $\phi$ on the alphabet $[k]$. By~\eqref{lacefunforsmallruns} and~\eqref{phirule}, we see that \begin{align*}\displaystyle\Phi\left( \prod_{i=1}^k  \exp\left(t \cdot\frac{x_i - x_i^m}{1-x_i^m}\right) \right) = \frac{1}{1 - \sum_{i=1}^k \frac{x_i - x_i^m}{1 - x_i^m}}\end{align*} gives the sum of all weights in $m$-Carlitz $k$-ary words.  This can also be found by substituting $x_i + x_i^2 + \ldots +x_i^{m-1}$ for $x_i$ in Lemma~\ref{carlitzgf}.

\section{Vincular patterns}

We are now ready to state a general formula for $k$-ary words avoiding vincular patterns with ones.  We say that a $k$-ary factorization $\phi$ contains a vincular pattern $\tau$ with only ones if the word made from $\phi$ by inserting a single $0$ between each pair of adjacent factors contains $\tau$, and this copy of $\tau$ does not use $0$.  Using the transformations $T$ and $\Phi$, we can reduce the problem to finding ordinary generating functions for factorizations that only use one symbol and avoid the given vincular pattern.

\begin{theorem}{\label{vincgf}}  Let $m_1, \ldots, m_n$ be positive integers, and let $A$ be the set of $k$-ary words avoiding the pattern $\tau = 1^{m_1} \d 1^{m_2} \d \cdots \d 1^{m_n}$.  Then \[ \sum_{a_1 \cdots a_l \in A} x_{a_1} x_{a_2} \cdots x_{a_l} =  \Phi\left(   \prod_{i=1}^k\left[ e^{tx_i} - T\left\{G_\tau(x_i, u)\right\}\right]\right)\] where $T$ is the operator defined in Section 3, and 
\begin{align}\label{G}
G_\tau(x,u) =\frac{ ux^{m_1}(1-x) }{(1-x - u(x-x^{m_i}))(1-x- ux)} \prod_{i=2}^n \left[ x^{m_i} + \frac{ux^{m_i}(1 - x^{m_i})}{1-x - u(x-x^{m_i})}  \right].
\end{align}
\end{theorem}
\begin{proof}
Earlier we observed that $e^{tx}$ is the Laguerre series for all factorizations using only a single letter with respect to the weight $w(\phi) = x^{\len(\phi)}$.  So it is enough to show that $G_\tau(x,u) = g_{A,w}(u)$ where $A$ is the set of factorizations on some single letter alphabet containing the given vincular pattern; the difference will then count the factorizations on a single letter which avoid the pattern.  For the remainder of the proof, we will use only the alphabet $\{1\}$.
We will use repeatedly the following basic facts:

\begin{fact}  Given two sets of factorizations $A_1$ and $A_2$, let $$A_3 = \{ (\phi_1)\ldots(\phi_l)(\psi_1)\ldots(\psi_m): \phi = (\phi_1)\ldots(\phi_l) \in A_1, \psi = (\psi_1)\ldots(\psi_m) \in A_2\}.$$  Then $g_{A_3,w}(u) = g_{A_1,w}(u) \cdot g_{A_2,w}(u)$ when both sides are defined.
\end{fact}

\begin{fact}Let $A$ be a set of factorizations on the alphabet $\{1\}$.  Then for a fixed integer $k\geq 0$, let $A_k$ be the set consisting of each factorization of $\phi$ with $k$ ones appended to the last factor: $A_k = \{(\phi_1)(\phi_2) \cdots (\phi_n 1^k): (\phi_1) \cdots (\phi_n) \in A\}$.  Then $g_{A_k,w}(u) = x^kg_{A,w}(u)$.\end{fact}

We say that a factorization {\it minimally} contains the pattern $1^{m_1} \d 1^{m_2} \d \cdots \d 1^{m_n}$ if it contains the pattern, but no longer contains the pattern if the last letter of the last factor is removed.  For example, $(11)(1)(1)(111)$ contains $111$ minimally, contains $11\d11$ but not minimally, and does not contain $1111$ at all.  Define a {\it truncation} of a factorization $\phi = (\phi_1) \cdots (\phi_n)$ to be a factorization $\hat{\phi} = (\phi_1) \cdots (\phi_{i-1}) (\hat{\phi_i})$, where $1 \leq i \leq n$ and $\hat{\phi_i}$ is a nonempty initial substring of $\phi_i$.  Any factorization containing a pattern has a unique truncation that contains that pattern minimally, so we will count the factorizations containing it minimally and then multiply by the appropriate expression to count all factorizations that contain the pattern.

First, suppose the pattern is just $1^{m_1}$.  A factorization minimally containing this pattern is any factorization whose last part has size $m_1$, and all other parts have size strictly smaller.  So the generating function for these factorizations is \[ \sum_{n=0}^\infty (u(x + \ldots + x^{m_1-1}))^nux^{m_1} = \left(\frac{ 1 - x}{1 - x - u(x - x^{m_1})}\right)ux^{m_1}.\]

Now let the generating function for the number of words minimally containing the pattern $1^{m_1} \d \cdots \d 1^{m_{n}}$ be $g_n(x,u)$; we will find $g_n(x,u)$ in terms of $g_{n-1}(x,u)$.  Let $\phi$ be a factorization minimally containing $1^{m_1} \d\cdots\d 1^{m_{n-1}} \d 1^{m_n}$; then $\phi$ has a unique truncation $\hat{\phi}$ that minimally contains  $1^{m_1} \d\cdots \d 1^{m_{n-1}}$. Thus to find $g_n(x,u)$ we multiply $g_{n-1}(x,u)$ by the appropriate factor, mimicking the following process by which any such $\phi$ is attained uniquely from some $\hat{\phi}$. If $\phi$ and $\hat{\phi}$ have the same number of factors, $\phi$ must be $\hat{\phi}$ with exactly $m_n$ $1$'s appended to the last factor; so we multiply by $x^{m_n}$.  Otherwise, we can assume a new factor is necessary.  In this case, we may lengthen the final part of $\phi$ by adding $1$'s, but to avoid having the pattern $1^{m_1} \d \cdots \d 1^{m_{n}}$ without adding any factors at most $m_n-1$ should be added.  This means we should first multiply by $$(1 + x + \ldots x^{m_n-1}) = \frac{1-x^{m_n}}{1-x}$$.  Then, we may append some (possibly empty) sequence of factors to $\hat{\phi}$, each having length less than $m_n$.  This amounts to multiplying by $$\sum_{i=0}^\infty (u(x + \ldots + x^{m_n - 1}))^i = \frac{ 1 - x}{1 - x - u(x - x^{m_{n})}}.$$ Finally, we append a new factor which is just $1^{m_n}$, multiplying by $ux^{m_n}$.   Thus we multiply by \[\frac{1-x^{m_n}}{1-x} \cdot \frac{ 1 - x}{1 - x - u(x - x^{m_{n})}}\cdot ux^{m_n}.  \]  So the ordinary generating function minimally avoiding $1^{m_1} -\ldots -1^{m_{n}}$ is exactly \[g_n(x,u) = g_{n-1}(x,u) \cdot \left(x^{m_n} + \frac{ ux^{m_n}(1 - x^{m_n})}{1 - x - u(x - x^{m_n})}\right).\]  Finally, once we have a factorization minimally containing the pattern, we may add any number of ones to the last factor and append any list of factors and the resulting factorization will still contain the pattern.  This amounts to multiplying by \[   \sum_{j=0}^\infty\sum_{i=0}^\infty x^j \left(u( x + x^2 + \ldots)\right)^i = \frac{1}{1-x} \frac{1 -x}{1 - x - ux}.\] Putting this all together, we see inductively that the generating function for factorizations containing the pattern is
\begin{align*}
\frac{ ux^{m_1}(1-x) }{(1-x - u(x-x^{m_i}))} \prod_{i=2}^n \left[ x^{m_i} + \frac{1 - x^{m_i}}{1-x - u(x-x^{m_i})}\right] \frac{1}{1 - x - ux}.
\end{align*}
\end{proof}

\section{Cyclically avoiding patterns}

We say that a word $W$ cyclically avoids a vincular pattern $\tau$ if $W$ avoids $\tau$ no matter how its letters are cycled.  More formally, let $r$ be the function that cycles $W$, moving the last letter into the first position: $r(a_1 \cdots a_n) = a_n a_1 \cdots a_{n-1} $.  Then $W$ cyclically avoids $\tau$ if $r^k(W)$ avoids $\tau$ for each $k$.

In order to find the generating function for the number of words cyclically avoiding the pattern $\tau = 1^m \d 1^m \d \cdots \d 1^m$, we will need a little more information than provided by the generating function $G_\tau(x,u)$ defined by ~\eqref{G}.  Let $H(x,u,v) = g_{A,w}(u)$ where $A$ is the set of factorizations on the alphabet $\{1\}$ {\it avoiding} the pattern $\tau$ and $\hat{w}$ is the weight $\hat{w}(\phi) = x^{\len(\phi)} u^{\parts(\phi)}  v^{ \fst(\phi)}$ where $\fst(\phi)$ is the size of the first factor of $\phi$.  Note that $\hat{w}$ is trivially a weight by our definition since we are using a singleton alphabet, but generally is not. We will find a  closed-form expression for $H(x,u,v)$, although it is rather unwieldy.

\begin{lemma}\label{hlemma} The generating function $H(x,u,v)$ is given by
\begin{align}
H(x,u,v) = 1 &+ \left[ \frac{1-x}{(1-vx)(1-x-ux)}\right] \biggl[{u(vx - (vx)^{mn})} +\biggl.\notag\\
& \left.\frac{{u^{2}x^{m} \left({\left(1-vx\right)} {\left(z - (vx)^{m}\right)} z^{n - 1} - {\left(1-(vx)^{m}\right)} {\left(z^n -(vx)^{m n}\right)}\right)} }{{\left(z-(vx)^{m}\right)} {\left({1-x - u\left(x - x^{m}\right)}\right)}}\right]\label{H}
\end{align}
where $$z = x^m + \frac{ux^m(1 - x^m)}{1 - x - u(x - x^m)}.$$   
\end{lemma}
\begin{proof}

Fix $n>0$, and let $\tau_n$ be the pattern $1^m \d 1^m \d \cdots \d 1^m$ with $n-1$ dashes.  Then by the proof of Theorem~\ref{vincgf}, we know that if $G_n(x,u)$ is the generating function~\eqref{G} for the number of factorizations on the alphabet $\{1\}$ containing $\tau_n$, then $$G_n(x,u) = \frac{ ux^{m}(1-x) }{(1-x - u(x-x^{m}))(1-x- ux)} \left(  x^{m} + \frac{ux^{m}(1 - x^{m})}{1-x - u(x-x^{m})}\right)^{n-1}$$  since in this case $m_1 = m_2 = \ldots = m.$  We have different cases depending on how much of the pattern $\tau_n$ is contained in the first part of a factorization.  Let $A$ be the set of all factorizations on the alphabet $\{1\}$ that contain $\tau_n$, and define define $A_0, A_1, \ldots, A_n$ by 
\begin{align*}
A_0 &= \{\phi =  (\phi_1)\cdots(\phi_l) \in A : \phi_1\text{ avoids }\tau_1 = 1^m\}\\
A_i &= \{\phi = (\phi_1)\cdots(\phi_l)\in A: \phi_1 \text{ contains }\tau_i \text{ but }\phi_1 \text{ avoids } \tau_{i+1}\} \text{  for }1 \leq i \leq n-1\}\\
A_n &= \{\phi = (\phi_1)\cdots(\phi_l)\in A : \phi_1 \text{ contains }\tau_n \}.
\end{align*}

First, suppose that $\phi \in A_0$.  Then the first part of $\phi$ is one of $1, 11, \ldots, 1^{m-1}$, and the factorization $(\phi_2)\cdots(\phi_k)$ of the remaining parts can be any factorization containing the pattern.  With the exponent of $v$ recording the size of the first factor, we see that the generating function corresponding to this case is  
\begin{align*}
g_{A_0,w}(u) &= u(vx +(vx)^2 + \ldots + (vx)^{m-1}) G_n(x,u) \\
&= u\cdot \frac{vx - (vx)^m}{1 - vx}  G_n(x,u).
\end{align*}

Now suppose that $\phi \in A_i$, $1 \leq i \leq n-1$. Then the first part of $\phi$ is one of $1^{im}, 1^{im+1}, \ldots, 1^{im + m-1}$ and $(\phi_2)\cdots(\phi_l)$ can be any factorization containing $\tau_{n-i}$.  So the generating function corresponding to $A_i$ is  
\begin{align*}
g_{A_i,w}(u) &= u((vx)^{im} + (vx)^{im+1} + \ldots + (vx)^{im + m -1}) G_{\tau_{n-i}}(x,u)\\
&=  u(vx)^{im}\frac{1 - (vx)^{m}  }{1 - vx}G_{n-i}(x,u).
\end{align*} 
If $\phi \in A_n$, then the first factor of $\phi$ can have any length that is $nm$ or more, and the remaining parts of $\phi$ can be any factorization.  Therefore 
\begin{align*}
g_{A_n,w}(u) &=  u((vx)^{nm} + (vx)^{nm+1} + \ldots)\sum_j(u(x +x^2 + \ldots))^j\\
&= \frac{  u(vx)^{nm}}{(1 - vx)(1 - \frac{ux}{1-x})}.
\end{align*}
To find the generating function for factorizations avoiding the pattern, we must subtract $g_{A_i,w}(u)$, $0 \leq i \leq n$, from the generating function for all one-letter factorizations, which is 
\begin{align*}
1 + u(vx + (vx)^2 + \ldots)\sum_{k=0}^\infty(u(x + x^2 + \ldots))^k = 1 + \frac{uvx}{(1-vx)(1 - \frac{ux}{1-x})}.\\
\end{align*}
We have
\begin{align*}
H(x,u,v) =&\, 1 + \frac{uvx}{(1-vx)(1 - \frac{ux}{1-x})} - \sum_{i=0}^n g_{A_i,w}(u)\\
=& \,1 + \frac{uvx}{(1-vx)(1 - \frac{ux}{1-x})} -  u\cdot \frac{vx - (vx)^m}{1 - vx}  G_{n}(x,u)- \frac{  u(vx)^{nm}}{(1 - vx)(1 - \frac{ux}{1-x})} \\
& -  \sum_{i=1}^{n-1} u(vx)^{im}\frac{1 - (vx)^{i}  }{1 - vx}G_{n-i}(x,u).
\end{align*}
After simplifying, we have the desired formula.
\end{proof}

\begin{theorem}\label{cyclicvinc}
Let $H(x,u,v)$ be the generating function~\eqref{H}.  Let $A$ be the set of words cyclically avoiding the pattern $\tau = 1^m \d 1^m \d\cdots \d 1^m$, with $n-1$ dashes, and let $w$ be the weight on $k$-ary words with $w(s_1 \ldots s_l) = x_{s_1} x_{s_2} \cdots x_{s_l}$.  Then the generating function $\sum_{W\in A} w(W)$ is given by

\[1 + \sum_{i=1}^k \Phi\biggl( t^{-1}\cdot T\left\{ u \frac{d^2}{dv\, du}\bigg|_{v=1}   H(x_i,u,v)\right\} \biggl(- 1 +\prod^k_{\substack{j=1\\
 j\neq i}} T\{H(x_j,u,1)\}\biggr)\biggr) + \sum_{i=1}^k \frac{x_i - x_i^{mn}}{1-x}.\]
In particular, letting $x_i = x$ for each $i$ gives:
\[\sum_{W \in A} x^{\len(W)} = 1 + k\cdot \Phi\biggl( t^{-1}\cdot T\left\{ u \frac{d^2}{dv\, du}\bigg|_{v=1}   H(x,u,v)\} \big(T\{H(x,u,1)\}\right)^{k-1}-1\big)\biggr) + \frac{k(x - x^{mn})}{1-x}.\]
\end{theorem}
\begin{proof}

We will count all words which avoid $\tau$ and begin with $i$, where $1 \leq i \leq k$, and then sum over $i$.  Let $\mathcal{W}$ be the set of all $k$-ary words $W$ that begin with $i$, do not end in $i$, and avoid the pattern $\tau$.  Any $W \in \mathcal{W}$ will in fact {\it cyclically} avoid $\tau$.  To see this, it is enough to note that the restriction $r^p(W)|_{\{j\}}$, for any $j,p$, will avoid the pattern.  Such a factorization will either be a permutation of the parts of $(W)|_{\{j\}}$, or will have some part of $(W)|_{\{j\}}$ divided into the first and last part of $r^p(W)|_{\{j\}}$. Permuting the factors will not cause $(W)|_{\{j\}}$ to contain the pattern by the symmetry of $\tau$, and splitting a part into two cannot create an instance of the pattern either.

We can characterize $\mathcal{W}$ as follows.  Suppose $\phi = (\phi_1)\ldots(\phi_p)$ is a factorization over the alphabet $[k]\backslash\{i\}$.  Then we can insert $i$'s appropriately in between factors so that the result starts with $i$ and avoids the pattern.  Given a factorization $\psi = (\psi_1)\ldots(\psi_p)$ which avoids the pattern and has the same number of parts as $\phi$ but only uses the symbol $i$, let $W_{\phi,\psi} = \psi_1 \phi_1 \psi_2 \phi_2\ldots \psi_p \phi_p$ be the concatenation of the alternating factors of $\phi$ and $\psi$.  Then $W_{\phi,\psi}$ start with $i$, will not end in $i$, and avoids $\tau$, so $W_{\phi,\psi} \in \mathcal{W}$.  Furthermore, if $W \in \mathcal{W}$, let $\phi$ be the restriction $W|_{[k]\backslash \{i\}}$, and $\psi = W|_{\{i\}}$.  Then $\phi$ and $\psi$ will have the same number of factors, and $W = W_{\phi,\psi}$.  So we see that $\mathcal{W}$ coincides with the set of $W_{\phi,\psi}$ where $\phi,\psi$ are nonempty factorizations over the alphabets $[k]\backslash \{i\}$ and $\{i\}$, respectively, which avoid $\tau$ and have the same number of parts.
 
The words $W \in \mathcal{W}$ we have described always end in a letter besides $i$; to get all words avoiding the pattern starting with $i$, we associate $W\in \mathcal{W}$ with a class $$\{W, r^{-1}W, r^{-2}W, \ldots,r^{-B(W)+1}W\},$$ where $B(W)$ is the number of adjacent $i$'s at the beginning of $W$. These rotated words will also cyclically avoid the pattern, as we have explained, and begin with at least one $i$.  Furthermore, any word $V$ beginning with $i$ and avoiding the pattern will be in the class of a unique word $W$: let $W = r^CV$ where $C$ is the number of $i$'s on the end of $V$, if any.  Therefore we will count all words $W$ that avoid the pattern and begin with $i$, but do not end with $i$, with the multiplicity $B(W)$, the number of $i$'s at the beginning of the word $W$.  This will account for all words beginning with $i$ that use at least one letter $j \neq i$.  So the weight of all $k$-ary words that begin with $i$, cyclically avoid the pattern $\tau$, and contain at least one letter $j \neq i$ is $$\sum_{W \in \mathcal{W}} B(W)w(W) = \sum_{\psi,\phi} \fst(\phi) w(\phi)w(\psi)$$ where the sum on the right-hand side is taken over all factorizations $\phi,\psi$ over the alphabets $[k]\backslash \{i\}$ and $\{i\}$, respectively, which avoid $\tau$ and have the same number of parts.

For any $1 \leq j \leq k$ let $A_j$ be the set of nonempty factorizations on the alphabet $\{j\}$ that avoid the pattern. Define
\[b_p = \sum_{{\substack{\phi \in A_1 * \cdots *\hat{A_i} * \cdots * A_k\\\parts(\phi) = p}}} w(\phi),\] where the $\hat{A_i}$ indicates that $A_i$ is removed, and let \[a_p = \sum_{\phi \in A_i, \parts(\phi) = p} \fst(\phi)x_i^{\len(\phi)}.\]  Then by the previous discussion, we see that $$\sum_{W \in \mathcal{W}} B(W) w(W) = \sum_{p= 1}^\infty a_p b_p$$ is the weight of all $k$-ary words avoiding the pattern $\tau$ which begin with $i$ and use at least one letter $j \neq i$.  It remains to evaluate the right-hand side.

We make use of the well-known fact that the generalized Laguerre polynomials obey the orthogonality relation
\begin{align}\label{lagortho}
\int_0^\infty x^\alpha e^{-x} L_i^{(\alpha)}(x) L_j^{(\alpha)}(x) \, dx = \frac{\Gamma(i + \alpha + 1)}{i!} \delta_{i,j}.  
\end{align}
which can be found in, e.g., \cite[p. 241]{ortho}.  Although this only holds in general when $\alpha > - 1$, it is true for $\alpha = -1$ if $i \geq 1$.   We have
\begin{align}\label{mylagortho}
\Phi(it^{-1} l_i(t) l_j(t)) = \delta_{i,j}
\end{align}
for $i\geq 1$. Thus for a function $f(t)$ that is a formal sum of the Laguerre polynomials $l_p(t)$, we can extract the coefficient of $l_p(t)$ by evaluating $\Phi( pt^{-1} l_p(t) f(t))$, for $p\geq 1$.  Furthermore, we see that if $f(t) = \sum_{p\geq 1} pa_p l_p(t)$ and $g(t) = \sum_{p\geq 1} b_p l_p(t)$, we have
\begin{align}\label{ab}
\Phi(t^{-1} f(t)g(t)) = \sum_{p\geq 1} a_p b_p
\end{align}
when both sides are well-defined.  We require $p\geq 1$ since the orthogonality relation~\eqref{mylagortho} fails when $i = 0$.

In this case, $\sum_{p \geq 1} b_p l_p(t)$ is the Laguerre series for {\it nonempty} $k$-ary words avoiding the pattern that do not use $i$; it is given by $\left(\prod_{j \neq i} T\{H(x_j,u,1)\}\right) - 1$.  To find $\sum_{p \geq 1} pa_p l_p(t)$, recall from the Lemma~\ref{hlemma} that $H(x,u,v) = \sum_{\phi \in A_i} x^{\len(\phi)} u^{\parts(\phi)}  v^ {\fst(\phi)}$.  So \[u \frac{d^2}{dv\, du}\bigg|_{v=1}   g(x_i,u,v)  =   \sum_{\phi \in A_i, \phi \neq \emptyset}  \fst(\phi)\cdot \parts(\phi) \cdot x_i^{\len(\phi)} u^{\parts(\phi)}.\] Then applying $T$, we see that \[T\left\{u \frac{d^2}{dv\, du}\bigg|_{v=1}H(x_i,u,v)\right\} = \sum_{p \geq 1} p a_p l_p(t).\]  Thus by~\eqref{ab},
\begin{align}\label{beginningwithi}
\Phi\left(t^{-1}\cdot T\left\{  u\frac{d^2}{dv\, du}\bigg|_{v=1}   H(x_i,u,v)\right\}\cdot \prod_{j \neq i} \bigl(T\{G(x_j,u)\}\bigr)\right)
\end{align}
is the generating function for the number of $k$-ary words that begin with $i$, use at least one $j\neq i$, and cyclically avoid the pattern.  This does not account for the words that use only the letter $i$ and avoid the pattern, which are the words $i^k$, with $k< mn$.  So we must add $x_i + x_i^2 + \ldots + x_i^{mn-1} = \frac{x_i - x_i^{mn}}{1 - x_i}$ to ~\eqref{beginningwithi} to get the full generating function for all words that begin with $i$ and cyclically avoid $\tau$.  Summing over $i$ and adding $1$ to represent the empty factorization gives \[1 + \sum_{i=1}^k \Phi\biggl( t^{-1}\cdot T\left\{ u \frac{d^2}{dv\, du}\bigg|_{v=1}   H(x_i,u,v)\right\} \biggl(- 1 +\prod^k_{\substack{j=1\\
 j\neq i}} T\{H(x_j,u,1)\}\biggr)\biggr) + \sum_{i=1}^k \frac{x_i - x_i^{mn}}{1-x}\] as desired.
\end{proof}

If we set $n=1$, considering $k$-ary words that cyclically avoid $1^m$, the formula simplifies considerably.  After some computation, which we omit here, we arrive at the following formula.

\begin{corollary}
Let $A$ be the set of nonempty $k$-ary words avoiding $1^m$.  Let $w$ be the weight $w(a_1 \cdots a_l) = x_{a_1} x_{a_2} \cdots x_{a_l}$.  Then 
\[\sum_{W \in A} w(W) = \sum_{i = 1}^k \frac{  x_i^{2m} - m x_i^{m + 1} + (m - 1)x_i^{m}}{ (x_i^{m} - 1)(x_i -1) } + \frac{\sum_{i=1}^k \frac{ (m - 1)x_i^{m+1} - m x_i^{m} + x_i}{(x_i^{m} - 1)^2} }{1 - \sum_{i =1}^k \frac{x_i^{m} - {x_i}}{ x_i^{m} - 1}}.\]
\end{corollary}

In the book by Heubach and Mansour \cite{Mansour1}, the authors define a cyclic Carlitz composition as a Carlitz composition so that the first and last parts are not equal; they ask (Research Direction 3.3) for the generating function for the number of cyclic Carlitz compositions.  If we let $k$ approach infinity, $m=2$, and $x_i = x^i$, we get the following.

\begin{corollary}
Let $A$ be the set of cyclic Carlitz compositions.  Then
\[ \sum_{W \in A} x^{\operatorname{sum}(W)} = \frac{\sum_{i=1}^\infty \frac{x^i}{(1 + x^i)^2}}{1 - \sum_{i=1}^\infty \frac{x^i}{1 + x^i}} + \sum_{i=1}^\infty \frac{x^{2i}}{1 + x^i}.\]
\end{corollary}

Setting $x_i = x$ in Corollary 5.3 and simplifying gives the following formula.

\begin{corollary}
Let $A$ be the set of nonempty $k$-ary words that cyclically avoid $1^m$.  Then \[\sum_{W \in A} x^{\len(W)} = \frac{1 - x^{m-1}}{ 1 - x} \left(kx + (k-1)x\left(\frac{m - (m-1)kx}{1 - kx + (k-1)x^m} - \frac{m}{1 - x^{m}}\right)\right).\]
\end{corollary}

This was found by Burstein and Wilf \cite{Wilf}. They go on to show that the number of $k$-ary words of length $n$ cyclically avoiding $1^m$ is asymptotically $\beta^n$, where $\beta$ is the positive root of $x^{m+1} = (k - 1)(1 + x + x^2 + \ldots + x^m)$. In fact, they extract an explicit formula when $n$ is sufficiently large.

We can also give a cyclic version of Theorem~\ref{fundamental}, which can be derived by extracting the coefficient of a monomial $x_1^{n_1}\cdots x_k^{n_k}$ in the generating function from Theorem 5.2 for words cyclically avoiding $1^m$.

\begin{corollary}
Let $n_1, \ldots, n_k$ be positive integers, and let $p_{m ,n}(t)$ be defined as before by $\sum_{n=0}^\infty p_{m,n}(t) \,x^n = \exp{\left( \frac{t(x - x^m)}{1- x^m}\right)}$.  Then \[N\cdot \Phi\left(t^{-1}\cdot \prod_{i =1}^k p_{m,n_i}(t)\right)\] is the total number of $k$-ary words that use the letter $i$ exactly $n_i$ times and cyclically avoid $1^m$, where $N = \sum_{i=1}^k n_i$ is the total number of letters counted with multiplicity.
\end{corollary}

\section{Questions and future directions}
There are a number of related questions that remain open.  Is it possible to extend the work of Section 5, finding generating functions for words avoiding other cyclic patterns? Is there a formula generalizing Theorem \ref{vincgf} or Theorem \ref{cyclicvinc} to to count the number of occurrences of a given pattern of ones?  One might also look for a combinatorial interpretation of some form of {\it composition} of Laguerre series; empirically, it seems that $l_i(l_j(t))$ has nonnegative integer coefficients in the $l_k$-basis for $j>0$. Finally, it would be useful to develop bijections from sets of words with restrictions to other combinatorial objects that are not obviously described in terms of words, using the methods outlined here to count sets that may be otherwise difficult to enumerate.

\bibliographystyle{plain}	


\end{document}